\documentclass[amscd,amssymb,12pt]{amsart}

\usepackage{graphicx}
\usepackage{tikz}
\usetikzlibrary{calc,arrows.meta,decorations.pathmorphing}
\usepackage[hidelinks]{hyperref}

\newtheorem*{Main}{Main Theorem}
\newtheorem{Thm}{Theorem}[section]
\newtheorem{Cor}[Thm]{Corollary}
\newtheorem{Lem}[Thm]{Lemma}
\newtheorem{Prop}[Thm]{Proposition}
\theoremstyle{definition}
\newtheorem{Def}[Thm]{Definition}
\theoremstyle{remark}
\newtheorem{Rem}[Thm]{Remark}

\DeclareMathOperator{\diam}{diam}

\DeclareMathOperator{\length}{length}

\begin{document}

\title{Additive quasi-isometries and cacti}
\author{Panos Papasoglu}
\address{Mathematical Institute, University of Oxford, Woodstock Road, Oxford OX2 6GG, U.K.}
\email{papazoglou@maths.ox.ac.uk}
\author{Eric Swenson}
\address{Department of Mathematics, Brigham Young University, 275 TMCB, Provo, UT 84602, U.S.A.}
\email{eric@mathematics.byu.edu}

\begin{abstract}
We prove that if a geodesic metric space contains no $c$-fat theta curve for
some $c>0$, then it is $(1,K)$-quasi-isometric to a cactus graph, where $K$
depends only on $c$.  Using a  coarse characterization of cacti  in terms of $c$-fat theta curves this implies that every geodesic metric space
quasi-isometric to a cactus is $(1,K)$-quasi-isometric to a cactus graph.
\end{abstract}

\maketitle

\section{Introduction}

Throughout the paper a \emph{cactus} means a connected graph in which any two
simple cycles have at most one vertex in common; graphs are endowed with their
path metrics.  A map $F:X\to Y$ is a
$(1,K)$-quasi-isometry if
\[
        \bigl|d_Y(F(x),F(y))-d_X(x,y)\bigr|\leq K
\]
for all $x,y\in X$, and every point of $Y$ lies at distance at most $K$ from
$F(X)$.

Studying spaces up to quasi-isometry originated in geometric group theory but there has been a recent interest in the subject by both computer scientists and graph theorists as it gives a natural perspective when studying geodesic metric spaces such as networks or graphs.

Manning \cite{Man} characterized geodesic metric spaces quasi-isometric to trees by the bottleneck property.  In the graph case, Georgakopoulos--Papasoglu \cite{GP} reformulated this in fat-minor language: a graph is quasi-isometric to a tree if and only if it excludes a sufficiently fat $K_3$-minor.  Kerr further proved that every geodesic metric space quasi-isometric
to a tree is $(1,K)$-quasi-isometric to a simplicial tree for some $K$
\cite[Theorem~1.3]{Kerr}; for graphs see also Berger--Seymour \cite{BS}.  Thus a graph excluding a sufficiently fat $K_3$-minor is $(1,K)$-quasi-isometric to a tree.

There were earlier results: in the graph case,
Chepoi--Dragan--Newman--Rabinovich--Vax\`es proved that if a graph $G$ embeds
in a tree with multiplicative distortion $\lambda$, then there is a tree $T$
on the same vertex set such that
\[
 d_T(x,y)-2\leq d_G(x,y)\leq d_T(x,y)+3\lambda
 \qquad(x,y\in G)
\]
\cite[Corollary~4]{CDNRV}. It is not hard to see that their result implies Kerr's result-at least for graphs.

Chepoi et al. proved further that if a graph does not contain an $r$-metric relaxed $K_{2,3}$-minor model (see Section~\ref{S:prelim} for a definition), then it is quasi-isometric to a cactus.  Fujiwara--Papasoglu \cite{FPcactus} later proved that if a geodesic metric space does not contain a $c$-fat theta curve for some $c$, then it is quasi-isometric to a cactus.  These two obstructions are equivalent up to a change of constant; in Section~\ref{S:prelim} we record the direction needed for the proof.

In this paper we strengthen this result and we show that under this condition the space is in fact $(1,K)$-quasi-isometric to a cactus.

For $c>0$, a \emph{$c$-fat theta curve} consists of two non-empty connected
sets $A,B$ and three paths from $A$ to $B$ such that $d(A,B)>3c$ and the
portions of the three paths outside $N_c(A)\cup N_c(B)$ are pairwise at
distance greater than $c$.  Fujiwara--Papasoglu use a slightly different,
equivalent definition of a fat theta curve; we show the equivalence in Lemma~\ref{L:theta-comparison} below.

The question of when the multiplicative constant in a quasi-isometry can be
removed has recently been formulated systematically by Nguyen--Scott--Seymour
\cite{NSS}.  Given a class $\mathcal G$ of connected graphs, they ask when the
following holds: for every $L,C$ there is $C'$ such that, whenever a graph is
$(L,C)$-quasi-isometric to a member of $\mathcal G$, it is
$(1,C')$-quasi-isometric to a member of $\mathcal G$.  They conjecture this
for classes closed under edge contraction and subdivision.  They prove strong
forms of this statement for graphs of bounded path-width, and more generally
bounded line-width, and observe that it also holds for the class of connected
outerplanar graphs.  The corresponding question for planar graphs remains
open. It is easy to see that a quasi-isometry can not always be upgraded to a $(1,K)$-quasi-isometry and
Davies,  Hatzel and Hickingbotham, \cite{DHH} have some stronger negative results limiting the scope of \cite{NSS}.

Our result is also naturally viewed from the perspective of coarse
graph theory.  Georgakopoulos--Papasoglu \cite{GP} introduced fat minors as
large-scale analogues of graph minors and formulated general questions asking
whether exclusion of sufficiently fat minors forces
quasi-isometry to a graph with the corresponding minors
excluded.  This turned out to be false in general \cite{DHIM} but for cacti one obtains even a $(1,K)$-quasi-isometry.

\begin{Main}
Let $X$ be a geodesic metric space.  If $X$ contains no $c$-fat theta curve
for some $c>0$, then there are a constant $K=K(c)$ and a cactus $Q$ such that
$X$ and $Q$ are $(1,K)$-quasi-isometric.
\end{Main}

\begin{Cor}\label{C:quasi-cactus}
If a geodesic metric space $X$ is quasi-isometric to a cactus, then there are
a constant $K$ and a cactus $Q$ such that $X$ and $Q$ are
$(1,K)$-quasi-isometric.
\end{Cor}

\begin{proof}
By the characterization of Fujiwara--Papasoglu and
Lemma~\ref{L:theta-comparison}, $X$ contains no $c$-fat theta curve for some
$c>0$.  The constants in that characterization are uniform, so $c$, and hence
$K(c)$, depends only on the quasi-isometry constants.  The Main Theorem applies.
\end{proof}

\begin{Cor}\label{C:outerplanar}
If a connected graph is quasi-isometric to a connected outerplanar graph, then
it is $(1,K)$-quasi-isometric to a cactus for some $K$.
\end{Cor}

\begin{proof}
By the outerplanar case of the Nguyen--Scott--Seymour result \cite{NSS}, the
graph is $(1,K_{\mathrm{op}})$-quasi-isometric to a connected outerplanar graph $H$, for
some $K_{\mathrm{op}}$.  Since outerplanar graphs are $K_{2,3}$-minor-free, $H$ contains no
$1$-fat theta curve: otherwise the two connected endpoint neighbourhoods and
the three separated middle subpaths give a $K_{2,3}$-minor.  The Main Theorem
gives a $(1,K_{\mathrm{c}})$-quasi-isometry from $H$ to a cactus.  Composing the two maps
gives the result.
\end{proof}

Since the class of cacti is closed under edge contraction and subdivision,
Corollary~\ref{C:quasi-cactus} verifies the Nguyen--Scott--Seymour principle
for this class.  Corollary~\ref{C:outerplanar} also strengthens their
outerplanar conclusion by allowing the additive target to be chosen to be a
cactus.  

Our proof uses the layering partition of
Chepoi--Dragan--Newman--Rabinovich--Vax\`es.  Instead of constructing a cactus with the same vertex set as the original graph as they do, we pass instead to a quotient graph which we show that is a cactus.  The quotient map preserves height from a basepoint.  It is also distance non-increasing so we only need to show that the projection does not shorten distances by much.  A geodesic in the quotient decomposes into two radial pieces and an arc of at most one cycle; the latter consists of at most two height-monotone pieces.  These pieces lift isometrically to the original graph, and to lift the whole geodesic we only need to add at most three bounded connecting paths inside the fibres.  This yields the result.

We used ChatGPT-5.6 Sol for the exposition of the results in \cite{CDNRV}. 

\section{Coarse obstructions}\label{S:prelim}

If $A,B$ are subsets of a metric space, write
\[
        d(A,B)=\inf\{d(a,b):a\in A,\ b\in B\}.
\]
For $r>0$, let $N_r(A)$ denote the closed $r$-neighbourhood of $A$.
A path whose endpoints are $x$ and $y$ will be called an $x$--$y$ path.

\begin{Def}\label{D:fat-theta}
Let $c>0$.  A \emph{$c$-fat theta curve} consists of two non-empty connected
sets $A,B$ and three paths $q_1,q_2,q_3$ joining $A$ to $B$ such that
\[
        d(A,B)>3c,
\]
and such that the portions of $q_1,q_2,q_3$ outside
$N_c(A)\cup N_c(B)$ are pairwise at distance greater than $c$.
\end{Def}

Fujiwara--Papasoglu use the earlier definition in which the three paths have
common endpoints and the initial and terminal pieces are separated from one
another by a prescribed amount.  The two definitions differ only by a change
of constant:

\begin{Lem}\label{L:theta-comparison}
If $X$ contains an $M$-fat theta curve in the sense of
\cite[Definition~1.7]{FPcactus}, then it contains an $(M/2)$-fat theta curve in
the sense of Definition~\ref{D:fat-theta}.  Conversely, if $X$ contains a
$c$-fat theta curve in the sense of Definition~\ref{D:fat-theta}, then it
contains a $(c/3)$-fat theta curve in the sense of
\cite[Definition~1.7]{FPcactus}.
\end{Lem}

\begin{proof}
For the first assertion let $\alpha_i,\beta_i$ be the initial and terminal
arcs in the Fujiwara--Papasoglu definition.  Put
\[
 A=\alpha_1\cup\alpha_2\cup\alpha_3,
 \qquad
 B=\beta_1\cup\beta_2\cup\beta_3.
\]
These sets are connected, $d(A,B)\geq2M>3M/2$, and the three remaining
middle arcs are pairwise at distance at least $M>M/2$.

Conversely, let $A,B,q_1,q_2,q_3$ be a $c$-fat theta curve.  Choose
$0<\varepsilon<c/6$.  For each $i$, let $v_i$ be the first point of $q_i$
in $N_c(B)$ and let $u_i$ be the last point of $q_i$ in $N_c(A)$ before
$v_i$.  Then the open subpath $\gamma_i$ from $u_i$ to $v_i$ lies outside
$N_c(A)\cup N_c(B)$.

Fix $a\in A$ and $b\in B$.  Since the open $\varepsilon$-neighbourhood of
a connected subset of a geodesic space is path connected (use arbitrarily
fine chains in the connected set and join consecutive points geodesically),
each $u_i$ can be joined to $a$ by a path
$\alpha_i\subset N_{c+\varepsilon}^{\circ}(A)$, and each $v_i$ can be
joined to $b$ by a path
$\beta_i\subset N_{c+\varepsilon}^{\circ}(B)$.  The paths may be chosen
so that the part of $\alpha_i$ outside $N_c(A)$, and the part of $\beta_i$
outside $N_c(B)$, have diameter less than $\varepsilon$.  For $i\neq j$
these small terminal pieces cannot meet the open middle arc of $q_j$, since
$u_i,v_i$ are limits of the corresponding middle portions and those portions
are pairwise more than $c$ apart.  Trimming at the last intersection with
$\alpha_i$ and the first intersection with $\beta_i$, if necessary, then
makes the open middle arcs disjoint from all initial and terminal arcs.  The
open middle arcs are still pairwise at distance greater than $c$, while
\[
 d\bigl(N_{c+\varepsilon}^{\circ}(A),
         N_{c+\varepsilon}^{\circ}(B)\bigr)
 >c-2\varepsilon>2c/3.
\]
Thus the three resulting paths form a $(c/3)$-fat theta curve in the
Fujiwara--Papasoglu sense.
\end{proof}

The following elementary reduction allows us to work with graphs.

\begin{Lem}\label{L:graph-reduction}
Let $X_0$ be a geodesic metric space containing no $c_0$-fat theta curve.
Let $\Gamma(X_0)$ be the unit-edge graph with vertex set $X_0$, in which two
vertices are adjacent whenever their distance in $X_0$ is at most one.  Then
$X_0$ and $\Gamma(X_0)$ are $(1,1)$-quasi-isometric.  Moreover, if $c$ is an
integer with
\[
        c\geq100,\qquad c>3c_0+3,
\]
then $\Gamma(X_0)$ contains no $c$-fat theta curve.
\end{Lem}

\begin{proof}
For vertices $x,y$,
\[
 d_{X_0}(x,y)\leq d_{\Gamma(X_0)}(x,y)\leq d_{X_0}(x,y)+1,
\]
and every point of $\Gamma(X_0)$ is within $1/2$ of a vertex.  Thus the
identity on the common vertex set is a $(1,1)$-quasi-isometry.

For the second assertion, realize each edge of $\Gamma(X_0)$ by a geodesic
segment in $X_0$.  The resulting map $r:\Gamma(X_0)\to X_0$ satisfies
\[
        d_{X_0}(r(p),r(q))\geq d_{\Gamma(X_0)}(p,q)-3.
\]
If $A,B,q_1,q_2,q_3$ were a $c$-fat theta curve in $\Gamma(X_0)$, the
connected sets $r(N_c(A))$ and $r(N_c(B))$ would be more than $3c_0$ apart.
Moreover, a point of $r(q_i)$ outside their $c_0$-neighbourhoods can only come
from the middle portion of $q_i$; the three such portions therefore remain
pairwise more than $c-3>c_0$ apart.  Their images form a $c_0$-fat theta
curve in $X_0$, a contradiction.
\end{proof}

From now on we replace the original space by the connected unit-edge graph
$X=\Gamma(X_0)$, equipped with its combinatorial metric, and use the integer
$c$ supplied by Lemma~\ref{L:graph-reduction}.

Two subsets of $X$ are called $r$-far if their distance is greater than
$r$.  Let $H$ be a finite graph.  A \emph{relaxed minor model} of $H$ in
$X$ assigns to every vertex $u$ of $H$ a connected branch set
$V_u\subset X$ and to every edge $e=uv$ a path $P_e$ joining $V_u$ to
$V_v$, so that distinct branch sets are disjoint, $P_e$ avoids every
branch set not corresponding to an endpoint of $e$, and paths corresponding
to non-incident edges are disjoint.  Paths corresponding to incident edges
are allowed to meet away from their common branch set.

The model is an \emph{$r$-metric relaxed minor model} if, whenever
$e=uv$ and $e'=u'v'$ are non-incident edges of $H$, the two sets
\[
        V_u\cup P_e\cup V_v,
        \qquad
        V_{u'}\cup P_{e'}\cup V_{v'}
\]
are $r$-far.  When $H=K_{2,3}$ we speak of an
\emph{$r$-metric relaxed $K_{2,3}$}.  This is the terminology of
\cite[Section~4]{CDNRV}.

\begin{Lem}\label{L:relaxed-theta}
If $X$ contains a $100c$-metric relaxed $K_{2,3}$, then $X$ contains a
$c$-fat theta curve.
\end{Lem}

\begin{proof}
Write the two parts of $K_{2,3}$ as
\[
        \{a_1,a_2\},\qquad \{b_1,b_2,b_3\}.
\]
Let $V_{a_1},V_{a_2},V_{b_1},V_{b_2},V_{b_3}$ be the branch sets of a
$100c$-metric relaxed model, and write $P_{ki}$ for the model path
corresponding to the edge $a_kb_i$.  For each $i$, join the endpoint of
$P_{1i}$ in $V_{b_i}$ to the endpoint of $P_{2i}$ in $V_{b_i}$ by a path
$T_i\subset V_{b_i}$.  In this way we obtain a path $Q_i$ from
$V_{a_1}$ to $V_{a_2}$.  Fix a point $z_i\in T_i$ and regard $Q_i$ as the
concatenation of a left subpath $Q_i^-$ from $V_{a_1}$ to $z_i$ and a right
subpath $Q_i^+$ from $z_i$ to $V_{a_2}$.

For $i\neq j$, the abstract edges $a_1b_j$ and $a_2b_i$ are non-incident.
Consequently
\[
        d(V_{a_1},V_{b_i})>100c.
\]
Similarly $d(V_{a_2},V_{b_i})>100c$.  Parameterize $Q_i^-$ and $Q_i^+$
from $z_i$ towards $V_{a_1}$ and $V_{a_2}$, respectively.  Among all pairs
$x\in Q_i^-$, $y\in Q_i^+$ with $d(x,y)\leq10c$, choose a pair
$x_i,y_i$ for which the sum of the two parameter distances from $z_i$ is
maximal.  Such a pair exists, and
\[
        d(x_i,y_i)=10c.
\]
Indeed, all distances between vertices are integral.  If the distance were
smaller than $10c$, one of the two points could be moved one edge farther away
from $z_i$ while preserving the inequality $d(x,y)\leq10c$.  The two outer
endpoints cannot both have been reached, since $V_{a_1}$ and $V_{a_2}$ are
$100c$-far by the metric-relaxed condition.

Let $L_i$ be the part of $Q_i^-$ from $V_{a_1}$ to $x_i$, and let $R_i$ be
the part of $Q_i^+$ from $y_i$ to $V_{a_2}$.  The maximal choice of
$x_i,y_i$ gives
\begin{equation}\label{E:self-tail-separation}
        d(L_i,R_i)=10c.
\end{equation}
For $i\neq j$, the sets $L_i$ and $R_j$ lie in the complete sets associated
to the non-incident edges $a_1b_i$ and $a_2b_j$, and therefore
\begin{equation}\label{E:cross-tail-separation}
        d(L_i,R_j)>100c.
\end{equation}
The same observation, using an index different from the one under
consideration, shows that $V_{a_1}$ is $100c$-far from every $R_i$, that
$V_{a_2}$ is $100c$-far from every $L_i$, and that
$d(V_{a_1},V_{a_2})>100c$.

Let
\[
 A=V_{a_1}\cup L_1\cup L_2\cup L_3,
 \qquad
 B=V_{a_2}\cup R_1\cup R_2\cup R_3.
\]
These sets are connected, and \eqref{E:self-tail-separation} and
\eqref{E:cross-tail-separation} give
\[
        d(A,B)=10c>3c.
\]
For each $i$, let $\gamma_i$ be a geodesic from $x_i$ to $y_i$.  Its length
is $10c$, so $\gamma_i$ is an $A$--$B$ path.  If $i\neq j$, every point of
$\gamma_i$ is within $10c$ of $y_i$, and every point of $\gamma_j$ is
within $10c$ of $x_j$.  Since $y_i$ and $x_j$ lie in the complete sets
associated to the non-incident edges $a_2b_i$ and $a_1b_j$,
\[
        d(\gamma_i,\gamma_j)>100c-20c>c.
\]
Thus the portions of the three paths $\gamma_1,\gamma_2,\gamma_3$ outside
$N_c(A)\cup N_c(B)$ are pairwise at distance greater than $c$.  They form a
$c$-fat theta curve.
\end{proof}

\begin{Prop}\label{P:obstructions}
Put
\[
        \lambda=100c.
\]
Then $X$ contains no $\lambda$-metric relaxed $K_{2,3}$.
\end{Prop}

\begin{proof}
Otherwise Lemma~\ref{L:relaxed-theta} would give a $c$-fat theta curve.
\end{proof}

Fix a base vertex $v\in X$ and write
\[
        f(x)=d(v,x).
\]
We call $f(x)$ the \emph{height} of $x$.

\begin{Rem}
The theta graph may be represented by $K_4^-$, and $K_{2,3}$ is a subdivision of it.  Consequently, up to a change of constants, the existence of a fat theta curve is equivalent to the existence of a fat $K_{2,3}$-minor (up to change of constants).  We leave the elementary verification to the reader.
\end{Rem}

\section{Layering partitions and large clusters}\label{S:layering}

We recall in this section the layering-partition
construction of Chepoi--Dragan--Newman--Rabinovich--Vax\`es
\cite{CDNRV}.  
\subsection{The layering tree}

For $n\geq0$ put
\[
        L_n=\{x\in X:d(v,x)=n\}.
\]
Two vertices $x,y\in L_n$ are equivalent if they can be joined by a path in
$X\setminus B(v,n-1)$.  The equivalence classes are called the
\emph{clusters} of the $n$-th layer.  For $n=0$ there is the single cluster
$\{v\}$.  If $C$ is a cluster, its level will be denoted by $h(C)$.

Two clusters are declared adjacent if some edge of $X$ has one endpoint in
each of them.  Distinct clusters in the same layer cannot be adjacent: such
an edge would itself join them outside $B(v,n-1)$.  Thus distinct adjacent
clusters lie in consecutive layers.

\begin{Def}\label{D:layering-tree}
The graph whose vertices are the clusters and whose edges join adjacent
clusters is called the \emph{layering tree} and is denoted by $\Gamma$.
\end{Def}

\begin{Prop}\label{P:layering-tree}
The layering tree $\Gamma$ is a rooted tree, with root $\{v\}$.  Every cluster
$C\subset L_n$, $n\geq1$, has a unique parent cluster in $L_{n-1}$.
\end{Prop}

\begin{proof}
Every vertex of $L_n$ has a neighbour in $L_{n-1}$ on a geodesic to $v$, so
every cluster has at least one adjacent cluster in the preceding layer.
Suppose that $x,y\in C\subset L_n$ have neighbours $x^-,y^-\in L_{n-1}$.
There is an $x$--$y$ path outside $B(v,n-1)$.  Adding the two edges
$x^-x$ and $yy^-$ gives an $x^-$--$y^-$ path outside $B(v,n-2)$.
Thus $x^-$ and $y^-$ lie in the same cluster of $L_{n-1}$, proving
uniqueness of the parent.  Since level decreases by one along every parent
edge, repeated passage to the parent reaches the root and no cycle is
possible.
\end{proof}

A cluster $D$ whose parent is $C$ will be called a \emph{child} of $C$.

Set
\[
        \Lambda=4\lambda+2.
\]

The following consequence of \cite[Proposition~4 and Corollary~8]{CDNRV} will be used to produce cycles from the clusters of the layering partition.

\begin{Prop}[Chepoi--Dragan--Newman--Rabinovich--Vax\`es]\label{P:no-three}
No cluster contains three vertices which are pairwise $\Lambda$-far.
\end{Prop}

\begin{proof}
The proof of \cite[Proposition~4]{CDNRV} constructs a $\lambda$-metric
relaxed $K_{2,3}$ from three vertices of one layering cluster which are
pairwise $(4\lambda+2)$-far.  This is impossible by
Proposition~\ref{P:obstructions}.
\end{proof}

Chepoi et al. state their results for finite graphs.  Their results
used here apply without change to arbitrary graphs.  Indeed,
Proposition~\ref{P:no-three} implies that every cluster has finite diameter.
If a cluster contains a $\Lambda$-far pair $c_1,c_2$, every other vertex is
within $\Lambda$ of at least one of $c_1,c_2$; if it contains no such pair,
its diameter is at most $\Lambda$.  Since graph distances are integer-valued,
if a cluster $C$ has finite diameter then there are $c_1,c_2\in C$ with
$d(c_1,c_2)=\diam C$.  Thus the diametral pairs used in \cite{CDNRV} may
also be chosen here, and the arguments used below do not require
local finiteness.

\subsection{Small and big clusters}

A cluster $C$ is called \emph{small} if $\diam C\leq\Lambda$.
A non-small cluster contains a $\Lambda$-far pair.  By
Proposition~\ref{P:no-three}, it contains no three pairwise $\Lambda$-far
vertices.  Such a cluster is called \emph{bifocal}.  Choose a diametral pair
$c_1,c_2\in C$.  Assign each vertex of $C$ to a nearest one of $c_1,c_2$,
breaking ties arbitrarily.  The resulting sets $C_1,C_2$ are the two
\emph{cells} of $C$.  A small cluster is regarded as having one cell.

A bifocal cluster is called \emph{big} if
\[
        \diam C>16\lambda+12.
\]
A bifocal cluster is $R$-\emph{separated} if the distance between its two
cells is greater than $R$.

If $D$ is a child of a bifocal cluster $C$, and $D$ is itself
bifocal, then $D$ is called \emph{spread} if both cells of $C$ are adjacent
to $D$.  Here two sets are
adjacent when an edge of $X$ has one endpoint in each set.

The next estimates are \cite[Lemmas~7 and 8]{CDNRV}.  Their short proofs are
included because the numerical separation of the two cells will be used
later.

\begin{Lem}\label{L:cell-diameter}
If $C$ is bifocal, every point of $C_i$ lies within $\Lambda$ of $c_i$;
in particular each cell has diameter at most $2\Lambda$.  If
$\diam C>12\lambda+6$, then
\[
        \diam C_1,\diam C_2\leq\Lambda
\]
and
\[
        d(C_1,C_2)>\diam C-2\Lambda-1.
\]
In particular, every big cluster is $(8\lambda+8)$-separated.
\end{Lem}

\begin{proof}
If $x\in C_1$ and $d(x,c_1)>\Lambda$, then the definition of $C_1$ gives
$d(x,c_2)\geq d(x,c_1)>\Lambda$.  Since $c_1,c_2$ are $\Lambda$-far, the
three vertices $x,c_1,c_2$ would be pairwise $\Lambda$-far, contrary to
Proposition~\ref{P:no-three}.  Thus every point of $C_1$ is within
$\Lambda$ of $c_1$, and $\diam C_1\leq2\Lambda$.  The same holds for
$C_2$.

Now assume $\diam C=d(c_1,c_2)>12\lambda+6$.  If $u\in C_1$ and
$v\in C_2$, then
\[
\begin{aligned}
 d(c_1,c_2)
 &\leq d(c_1,u)+d(u,v)+d(v,c_2)\\
 &\leq 2\Lambda+d(u,v).
\end{aligned}
\]
Hence $d(u,v)\geq\diam C-2\Lambda$ and, since graph distances are
integral,
\[
        d(u,v)>\diam C-2\Lambda-1.
\]
Since this quantity is at least $\Lambda$, two points in one cell which
were $\Lambda$-far would, together with the opposite focus, contradict
Proposition~\ref{P:no-three}.  Hence each cell has diameter at most
$\Lambda$.  Substituting the definition of a big cluster gives the last assertion.
\end{proof}

\subsection{Propagation and non-branching}

The next facts are \cite[Lemmas~9--11]{CDNRV}.

\begin{Prop}[Chepoi--Dragan--Newman--Rabinovich--Vax\`es]\label{P:propagation}
\hfill\break
\begin{enumerate}
\item If $C$ is big, then it has a bifocal spread child $D$.  The cells can
be labelled $C_1,C_2$ and $D_1,D_2$ so that $C_1$ is adjacent to $D_1$,
$C_2$ is adjacent to $D_2$, and there are no cross adjacencies
$C_1$--$D_2$ or $C_2$--$D_1$.
\item If $D$ is big and $C$ is its parent, then $C$ is bifocal and the two
cells of $D$ have neighbours in different cells of $C$.  In particular,
every big cluster is spread relative to its parent.
\item If $C$ is big, then no cell of any child of $C$ is adjacent to both
cells of $C$.
\end{enumerate}
\end{Prop}

\begin{proof}
These are Lemmas~9, 10 and 11, respectively, of \cite{CDNRV}.  Their proofs
use only the cell-diameter and separation estimates of
Lemma~\ref{L:cell-diameter} and the definition of a layering cluster, and
therefore apply to an infinite graph.
\end{proof}

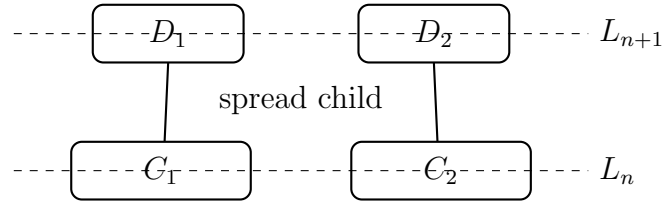
\begin{figure}[ht]
\centering
\begin{tikzpicture}[scale=.95]
  \draw[dashed] (-4,2.5)--(4,2.5) node[right] {$L_n$};
  \draw[dashed] (-4,4.4)--(4,4.4) node[right] {$L_{n+1}$};
  \draw[rounded corners,thick] (-3.2,2.1) rectangle (-.7,2.9);
  \draw[rounded corners,thick] (.7,2.1) rectangle (3.2,2.9);
  \node at (-1.95,2.5) {$C_1$};
  \node at (1.95,2.5) {$C_2$};
  \draw[rounded corners,thick] (-2.9,4.0) rectangle (-.8,4.8);
  \draw[rounded corners,thick] (.8,4.0) rectangle (2.9,4.8);
  \node at (-1.85,4.4) {$D_1$};
  \node at (1.85,4.4) {$D_2$};
  \draw[thick] (-1.9,2.9)--(-1.85,4.0);
  \draw[thick] (1.9,2.9)--(1.85,4.0);
  \node at (0,3.48) {spread child};
\end{tikzpicture}
\caption{The two cells of a big cluster propagate to the two cells of its
spread child.  Cross adjacencies are excluded.}
\label{F:spread-child}
\end{figure}

The proof of \cite[Proposition~5]{CDNRV} constructs a $\lambda$-metric
relaxed $K_{2,3}$ whenever one of the two branching configurations below
occurs.

\begin{Prop}\label{P:nonbranching}
\begin{enumerate}
\item Every cluster has at most one big child.
\item Every big cluster has at most one spread child.
\end{enumerate}
\end{Prop}

\begin{proof}
If a cluster has two big children, the proof of
\cite[Proposition~5]{CDNRV} constructs a $\lambda$-metric relaxed
$K_{2,3}$.  The same construction applies if a big cluster has two spread
children.  Both alternatives contradict Proposition~\ref{P:obstructions}.
\end{proof}

\section{The quotient cactus}\label{S:quotient}

Put
\[
        B=16\lambda+12.
\]
Thus every non-big cluster has diameter at most $B$, while every cell of a
big cluster has diameter at most $\Lambda$ by Lemma~\ref{L:cell-diameter}.

\subsection{Construction of the quotient}

Consider a maximal path in the layering tree consisting entirely of big
clusters.  By Proposition~\ref{P:nonbranching}, it has the form
\[
        C_r,C_{r+1},\ldots
\]
with $C_{i+1}$ the unique big child of $C_i$.  Such a path is finite.  Indeed,
choose points in the two cells of $C_r$.  Since they lie in the same layering
cluster, they can be joined by a finite path $P$ in
$X\setminus B(v,r-1)$.  Suppose a subpath of $P$ joins the two cells of $C_i$.  A maximal excursion
of this subpath outside $C_i$ has endpoints in different cells of $C_i$.
Since the subpath lies outside $B(v,i-1)$, the excursion cannot enter the
parent of $C_i$, and hence it enters a child $D$ of $C_i$.  Since the layering
graph is a tree, the whole excursion lies in the subtree rooted at $D$, so
$D$ is adjacent to both endpoint cells of $C_i$.  Proposition~\ref{P:propagation}(3)
then implies that $D$ is bifocal and that these two adjacencies occur in
different cells of $D$; thus $D$ is spread.  On the other hand the big child
$C_{i+1}$ is spread by Proposition~\ref{P:propagation}(2), so
Proposition~\ref{P:nonbranching}(2) gives $D=C_{i+1}$.  Thus $P$ contains a
subpath in $X\setminus B(v,i)$ joining the two cells of $C_{i+1}$.  Repeating this would force the finite path $P$ to
meet arbitrarily high layers if the maximal path of big clusters were
infinite, a contradiction.  Thus every maximal path of big clusters may be
written
\[
        C_r,C_{r+1},\ldots,C_s.
\]
Along it the two cells are labelled coherently as $C_i^+$ and $C_i^-$, using
Proposition~\ref{P:propagation}.

For every non-big cluster $C$ introduce one vertex $\bar C$.  For every big
cluster $C$, introduce two vertices $\bar C^+$ and $\bar C^-$ corresponding
to its two coherently labelled cells.  All edges introduced below have length
one.  If $D$ is a child of $C$, add edges according to the following four
rules.
\begin{enumerate}
\item If neither $C$ nor $D$ is big, add the edge $\bar C\bar D$.
\item If $C$ is not big and $D$ is big, add the two edges
\[
        \bar C\bar D^+,
        \qquad
        \bar C\bar D^-.
\]
\item If $C$ is big and $D$ is not big, then either $D$ is spread, in which
case add
\[
        \bar C^+\bar D,
        \qquad
        \bar C^-\bar D,
\]
or $D$ is adjacent to exactly one cell $C^\epsilon$ of $C$, in which case
add only $\bar C^\epsilon\bar D$.
\item If both $C$ and $D$ are big, label their cells coherently and add
\[
        \bar C^+\bar D^+,
        \qquad
        \bar C^-\bar D^-.
\]
\end{enumerate}
These are exactly the adjacencies between the corresponding pieces of the
clusters.  Indeed, Rule~2 follows from Proposition~\ref{P:propagation}(2),
Rule~3 from the definition of spread together with
Proposition~\ref{P:propagation}(3), and Rule~4 from
Proposition~\ref{P:propagation}(2)--(3).  Denote the resulting graph by $Q$.

Define
\[
        p:X\longrightarrow Q
\]
by
\[
 p(x)=\bar C
 \quad\hbox{if $x$ belongs to a non-big cluster $C$},
\]
and
\[
 p(x)=\bar C^\epsilon
 \quad\hbox{if $x\in C^\epsilon  \, (\epsilon=\pm),$  and $C$ is big}.
\]

\begin{Prop}\label{P:Q-cactus}
The graph $Q$ is a cactus.
\end{Prop}

\begin{proof}
Let
\[
        \mathcal C=(C_r,C_{r+1},\ldots,C_s)
\]
be a maximal path of big clusters.  The parent $P$ of $C_r$ is not big.  By
Proposition~\ref{P:propagation}(1), $C_s$ has a spread child $D$, which is
unique by Proposition~\ref{P:nonbranching}(2).  The cluster $D$ is not big,
since otherwise the path could be extended.  The four rules therefore give
two internally disjoint paths
\[
 \bar P-\bar C_r^+-\bar C_{r+1}^+-\cdots-\bar C_s^+-\bar D
\]
and
\[
 \bar P-\bar C_r^--\bar C_{r+1}^--\cdots-\bar C_s^--\bar D.
\]
Their union is a simple cycle, denoted by $Q_{\mathcal C}$.  We call $\bar P$
and $\bar D$ respectively the \emph{bottom} and \emph{top} of this cycle; the
two displayed paths are its \emph{sides}.

Every child of a cluster in $\mathcal C$ which does not lie in $\mathcal C$ is
attached to only one of the two sides, except for the spread child $D$ of
$C_s$.  Indeed, while a big child is present it is spread by
Proposition~\ref{P:propagation}(2), and Proposition~\ref{P:nonbranching}(2)
excludes any second spread child.  Thus, starting from the layering tree, the
passage to $Q$ replaces each path
\[
        P,C_r,\ldots,C_s,D
\]
associated to a maximal path of big clusters by two parallel paths with the
same endpoints, while every edge leaving an interior vertex is attached to
only one of the two copies.  Since the layering graph is a tree, each such
replacement creates exactly the one cycle $Q_{\mathcal C}$ and cannot create
a further cycle together with other replacements.  The graph $Q$ is connected,
and distinct maximal big paths are disjoint in the layering tree, so their
associated cycles can meet only at a non-big endpoint.  Hence any two simple
cycles of $Q$ have at most one vertex in common, and $Q$ is a cactus.
\end{proof}

Let $\bar v$ be the vertex corresponding to the root cluster $\{v\}$.  If a
vertex $z$ of $Q$ corresponds to a cluster in $L_n$, put
\[
        \bar f(z)=n
\]
and call $\bar f(z)$ the \emph{height} of $z$.

\begin{Lem}\label{L:quotient-basic}
For every $x\in X$,
\[
        d_Q(\bar v,p(x))=d_X(v,x).
\]
If $x,y$ are adjacent in $X$, then
\[
        d_Q(p(x),p(y))\leq1.
\]
Consequently
\[
        d_Q(p(x),p(y))\leq d_X(x,y)
        \qquad(x,y\in X).
\]
Moreover every point of $Q$ is at distance at most $1/2$ from $p(X)$.
\end{Lem}

\begin{proof}
Every edge of $Q$ joins vertices whose heights differ by one.  Hence any path
from $\bar v$ to a vertex of height $n$ has length at least $n$.
Conversely, let $x\in L_n$ and let
\[
        v=x_0,x_1,\ldots,x_n=x
\]
be a geodesic in $X$.  Consecutive vertices lie in consecutive layers, so
$p(x_0),\ldots,p(x_n)$ is a path in $Q$.  Its vertices are distinct because
they have distinct heights.  Thus it has length $n$, proving the first
assertion.

If $x,y$ lie in the same cluster and that cluster is big, they cannot lie in
different cells because the two cells are more than one apart.  Hence an edge
of $X$ either collapses to a point or, by the definition of $Q$, projects to
an edge.  This proves the second assertion and therefore the non-expanding
inequality.  Finally $p$ is onto the vertex set of $Q$, and every edge has
length one.
\end{proof}

\begin{Lem}\label{L:fibres-new}
For every vertex $z$ of $Q$,
\[
        \diam p^{-1}(z)\leq B.
\]
\end{Lem}

\begin{proof}
If $z=\bar C$ for a non-big cluster, then $p^{-1}(z)=C$ and
$\diam C\leq B$.  If $z=\bar C^\epsilon$ for a big cluster, then
$p^{-1}(z)=C^\epsilon$ and Lemma~\ref{L:cell-diameter} gives
$\diam C^\epsilon\leq\Lambda<B$.
\end{proof}

\subsection{Lifting geodesics in the quotient}

Every cycle $Q_{\mathcal C}$ associated to a maximal path of big clusters has
two sides on each of which the height increases by one from the bottom to the
top.  We call a path
\emph{height-monotone} if the height changes by one with the same sign along
every edge.

\begin{Lem}\label{L:monotone-lift}
Let $R$ be a height-monotone subpath of one side of a cycle
$Q_{\mathcal C}$.  If $a,b$ are its endpoints, then there are
$x_a\in p^{-1}(a)$ and $x_b\in p^{-1}(b)$ joined in $X$ by a path of length
exactly $\length(R)$.
\end{Lem}

\begin{proof}
Reverse $R$ if necessary, so that height increases along it.  Write the
maximal big path defining the cycle as $\mathcal C=(C_r,\ldots,C_s)$ and
suppose first that the upper endpoint of $R$ is a side vertex
$\bar C_j^\epsilon$.  Choose
$x_b\in C_j^\epsilon$ and a geodesic from $v$ to $x_b$.  Whenever this
geodesic passes from $C_{k+1}$ to its parent $C_k$, the no-cross-adjacency
part of Proposition~\ref{P:propagation}(1) forces it to remain in the
$\epsilon$-cell.  Its subpath between the heights of $a$ and $b$ therefore
projects exactly to $R$ and has length $\length(R)$.

It remains to consider the case in which the upper endpoint is the top
$\bar D$ of the cycle.  Choose an edge $yd$ of $X$ with
$y\in C_s^\epsilon$ and $d\in D$ on the required side.  A geodesic from
$v$ to $y$, followed by the edge $yd$, is a geodesic from $v$ to $d$, since
$f(d)=f(y)+1$.  The same no-cross-adjacency argument shows that its relevant
subpath projects exactly to $R$.
\end{proof}

\begin{Cor}\label{C:cycle-arc-lift}
Let $a,b$ be vertices of one cycle of $Q$.  Then there are
$x_a\in p^{-1}(a)$ and $x_b\in p^{-1}(b)$ such that
\[
        d_X(x_a,x_b)\leq d_Q(a,b)+B.
\]
\end{Cor}

\begin{proof}
Since $Q$ is a cactus, a shortest path between two vertices of one cycle
cannot leave that cycle and re-enter it at another vertex.  If the cycle
comes from the maximal big path $C_r,\ldots,C_s$, each of its two sides has
length $s-r+2$, since the height increases by one along every edge from the
bottom to the top.  Hence a geodesic arc from $a$ to $b$ is either a
height-monotone subpath of one side, or it is the union of two
height-monotone subpaths meeting at the bottom or at the top of the cycle.
In the first case Lemma~\ref{L:monotone-lift} gives the assertion with no
error.  In the second case lift the two subpaths separately.  Their two
lifted endpoints at the turning vertex lie in the same fibre, and
Lemma~\ref{L:fibres-new} joins them by a path of length at most $B$.
\end{proof}

\begin{Lem}\label{L:geodesic-decomposition}
Let $z,z'$ be vertices of $Q$, and choose geodesics $\alpha,\alpha'$ from
$\bar v$ to $z,z'$.  There are vertices $a\in\alpha$ and $b\in\alpha'$ such
that either $a=b$, or $a,b$ lie on one cycle of $Q$, and
\begin{equation}\label{E:decomposition-new}
 d_Q(z,z')=d_Q(z,a)+d_Q(a,b)+d_Q(b,z').
\end{equation}
\end{Lem}

\begin{proof}
Let $T$ be the block--vertex incidence tree of the cactus $Q$: its nodes are
the vertices and the blocks of $Q$, and a vertex-node is joined to each block
which contains it.  Root $T$ at the vertex-node $\bar v$.  The two rooted
paths from $\bar v$ to the vertex-nodes $z,z'$ have a last common node.  If
this is a vertex-node, take $a=b$ equal to that vertex.  Otherwise it is a
cycle block $S$; an edge block cannot be a branching last common node.  Let
$a$ be the vertex of $S$ through which the branch containing $z$ is attached,
taking $a=z$ when $z\in S$, and define $b$ similarly.  The chosen root
geodesics pass through $a$ and $b$.  Every path from $z$ to $z'$ must contain
an $a$--$b$ path.  If an $a$--$b$ path left $S$ and returned at a different
vertex, it would create a second cycle meeting $S$ in two vertices.  Thus the
shortest $a$--$b$ path is the shorter arc of $S$, and
\eqref{E:decomposition-new} follows.
\end{proof}

\begin{Thm}\label{T:metric-comparison}
Put
\[
        K=3B=48\lambda+36.
\]
Then, for all $x,y\in X$,
\[
        d_Q(p(x),p(y))
        \leq d_X(x,y)
        \leq d_Q(p(x),p(y))+K.
\]
Consequently $p:X\to Q$ is a $(1,K)$-quasi-isometry.
\end{Thm}

\begin{proof}
The first inequality and the $1/2$-density of $p(X)$ follow from
Lemma~\ref{L:quotient-basic}.  Put $z=p(x)$ and $z'=p(y)$.  Choose geodesics
$\gamma_x,\gamma_y$ in $X$ from $v$ to $x,y$.  By
Lemma~\ref{L:quotient-basic}, their projections are geodesics from $\bar v$
to $z,z'$ in $Q$.

Choose $a,b$ on these projected geodesics by
Lemma~\ref{L:geodesic-decomposition}, and let $x_a\in\gamma_x$, $y_b\in\gamma_y$ be
the vertices which project to $a,b$.  Since height changes by one along the
geodesics,
\[
 d_X(x,x_a)=d_Q(z,a),
 \qquad
 d_X(y,y_b)=d_Q(z',b).
\]
If $a=b$, Lemma~\ref{L:fibres-new} gives
\[
 d_X(x,y)
 \leq d_Q(z,a)+B+d_Q(b,z')
 =d_Q(z,z')+B.
\]

Suppose that $a\neq b$.  By the decomposition, $a,b$ lie on one cycle.
Corollary~\ref{C:cycle-arc-lift} gives points
$a'\in p^{-1}(a)$ and $b'\in p^{-1}(b)$ such that
\[
        d_X(a',b')\leq d_Q(a,b)+B.
\]
Using Lemma~\ref{L:fibres-new} once in each of the fibres over $a$ and $b$,
\[
\begin{aligned}
 d_X(x,y)
 &\leq d_Q(z,a)+B+d_Q(a,b)+B+B+d_Q(b,z')\\
 &=d_Q(z,z')+3B.
\end{aligned}
\]
This proves the second inequality.  Since $K>1/2$, the density assertion also
shows that $p$ is a $(1,K)$-quasi-isometry.
\end{proof}

\begin{proof}[Proof of the Main Theorem]
By Lemma~\ref{L:graph-reduction}, $X_0$ is $(1,1)$-quasi-isometric to a graph
$X$ containing no $c$-fat theta curve, where $c$ depends only on $c_0$.
Theorem~\ref{T:metric-comparison} gives a $(1,K(c))$-quasi-isometry from $X$
to a cactus, and composing the two quasi-isometries proves the theorem.
\end{proof}

\end{document}